\documentclass[oneside, 12pt]{amsart} \topmargin      -10mm \textwidth      160 true mm
\usepackage{amsmath, amssymb, amsfonts, amstext, amsthm, amscd}
\usepackage[mathscr]{euscript}
\usepackage{enumerate}
\usepackage{tikz,color,soul}
\usepackage[english]{babel}
\usepackage{chngcntr}
\usepackage{stmaryrd}
\usetikzlibrary{arrows.meta, positioning,arrows}
\newtheorem{theorem}{Theorem}[section]
\newtheorem{ques}[theorem]{Question}
\newtheorem{prop}[theorem]{Proposition}
\newtheorem{lem}[theorem]{Lemma}
\newtheorem{cor}[theorem]{Corollary}
\theoremstyle{definition}
\newtheorem{definition}[theorem]{Definition}
\newtheorem{eg}[theorem]{Example}
\newtheorem{remark}[theorem]{Remark}

\numberwithin{equation}{section}
\title{Extensions of i-reversible rings}
\address{Department of Mathematics, St. Joseph's College (Autonomous), Bangalore, India}
\email{v.bhabani.lama@gmail.com}
\author{Vivek Bhabani Lama, Suhas B N, Susobhan Mazumdar and Raisa DSouza }
\address{Department of Mathematics, St. Joseph's College (Autonomous), Bangalore, India}
\email{suhas.b.n@sjc.ac.in}
\address{Department of Mathematics, St. Joseph's College (Autonomous), Bangalore, India}
\email{susobhan@sjc.ac.in}
\address{Department of Mathematics, St. Joseph's College (Autonomous), Bangalore, India}
\email{raisadsouza@sjc.ac.in}
\date{}

\begin{document}
	
	\begin{abstract}
		A ring $R$ is said to be i-reversible if for every $a,b$ $\in$ $R$, $ab$ is a non-zero idempotent implies $ba$ is an idempotent. It is known that the rings $M_n(R)$ and $T_n(R)$ (the ring of all upper triangular matrices over $R$) are not i-reversible for $n \geq 3$. In this article, we provide a non-trivial i-reversible subring of $M_n(R)$ when $n \geq 3$ and $R$ has only trivial idempotents. We further provide a maximal i-reversible subring of $T_n(R)$ for each $n\geq 3$, if $R$ is a field. We then give conditions for i-reversibility of Trivial, Dorroh and Nagata extensions. Finally, we give some independent sufficient conditions for i-reversibility of polynomial rings, and more generally, of skew polynomial rings. 
	\end{abstract}
	
	\maketitle
  
	\section{Introduction}
 A ring $R$ is said to be reversible if for every $a,b\in R$, $ab=0$ then $ba=0$. In 1999,  Cohn \cite{cohn1999reversible}  formally introduced the notion of a reversible ring. Anderson and Camillo  \cite{anderson1999semigroups}, studied reversible rings under the nomenclature of $ZC_2$, while Krempa and Niewieczerzal \cite{krempa1977rings}, called them $C_0$ rings. It is easy to see that a ring $R$ is reversible if and only if for every $a,b\in R$, $ab$ is an idempotent implies $ba$ is an idempotent.
 In 2020, Anjana Khurana and Dinesh Khurana \cite{khurana2020reversible}, studied a new class of rings where $ab$ is a non-zero idempotent implies $ba$ is an idempotent and called them i-reversible rings. It was shown that the class of reversible rings is a proper subclass of the class of i-reversible rings.  Jung et.al \cite{jung2019reversibility}  studied the same class of rings under the terminology of quasi-reversible rings.
 
 A complete characterization of i-reversibility of matrix rings over commutative rings was established  in \cite[Theorem~4.3]{khurana2020reversible}. It was proved that for a commutative ring $R$, the matrix ring $M_n(R)$ is i-reversible if and  only if $n=2$ and $R$ has only trivial idempotents. It was also shown in \cite[Corollary~3.2]{khurana2020reversible}  that the upper triangular matrix ring $T_n(R)$ is not i-reversible  if $n > 2$. One may ask if there exists at least a non-trivial i-reversible subring of $T_n(R)$. We prove that the ring $D_n(R)=\big\lbrace (a_{ij}) \in T_n(R) |\ a_{11} =a_{22} = \cdots = a_{nn} \big\rbrace$ is such an example for any $n$, provided $R$ has only trivial idempotents (see Theorem \ref{P4.7} and Remark \ref{rr}). We also prove that for $n \geq 5$, if $S$ is a subring of $T_n(R)$ such that $D_n(R) \subsetneqq S$ then $S$ necessarily has a non-trivial idempotent and $S$ is not i-reversible (see Theorem \ref{TH3.7} and Theorem \ref{TH3.4}). Further, when $n \geq 5$ and $R$ is a field, we prove that $D_n(R)$ is a maximal i-reversible subring of $T_n(R)$ (see Theorem \ref{TH3.8}). Finally, when $n=3$ or $4$ and $R$ is a field, we prove that $D_n(R)$ is not a maximal i-reversible subring of $T_n(R)$ (see Examples \ref{EG1} and \ref{EG2}). We also provide maximal i-reversible subrings of $T_n(R)$ strictly containing $D_n(R)$ in these two cases (see Theorems \ref{TRT} and \ref{TRTT}).
 
In 2003, Kim and Lee \cite{kim2003extensions}  show that the polynomial ring over a reversible ring need not be reversible. They show however that if $R$ is also Armendariz then $R[x]$ is reversible. In the same article, they provide some sufficient conditions for reversibility of Dorroh, Nagata and Trivial extensions of a ring. In this paper, we discuss similar questions for i-reversibility. Analogous to \cite{kim2003extensions}, we show that if $R$ is i-reversible and Armendariz then $R[x]$ is i-reversible (see Theorem \ref{T3.21}). We also provide another sufficient condition for the i-reversibility of $R[x]$ (see Theorem \ref{P3.23}). Further, we prove that for Dorroh extensions, under suitable conditions, the reversibility and i-reversibility coincide (see Theorem \ref{P4.19}). Furthermore, we prove if $R$ has only trivial idempotents then the Trivial and Nagata extensions of $R$ over itself are i-reversible (see Theorems \ref{P4.14} and  \ref{P4.21}). We also provide a necessary and sufficient condition for the Trivial extension of $R$ to be i-reversible when $R$ is abelian and not reversible (see Theorem \ref{THM!!!!!}). 

An interesting generalization of  polynomial rings is the skew polynomial rings. The reversibility of skew polynomial rings has been studied in \cite{jin2017commutativity,pourtaherian2011skew}. We provide two independent sufficient conditions for the i-reversibility of skew polynomial rings (see Theorems \ref{T3.29} and \ref{THMM}).

 Throughout the course of the paper, all rings are considered to be associative rings with unity (unless specified otherwise). 
	\section{Trivial Extensions.}\label{S2}
	In this section, we show that if $R$ has only trivial idempotents then the trivial extension of $R$ over itself is i-reversible. We also provide examples to show that this condition is not necessary. However, if $R$ is abelian and not reversible, it turns out that this condition is necessary as well.
	
	\begin{definition}
		Let $R$ be a ring and $M$ be a bimodule of $R$. The trivial extension of $R$ by $M$ is the ring $T(R,M)=R\bigoplus M$ with component wise addition and the following multiplication:
		$$(r_1,m_1)\cdot(r_2,m_2)=(r_1r_2,r_1m_2+m_1r_2).$$
	\end{definition}

	\noindent The ring $T(R,M)$ is isomorphic to the ring of all matrices of the form
	$\begin{bmatrix}
	r&m\\
	0&r
	\end{bmatrix}$,
	where $r \in R$ and $m \in M$ with usual matrix operations. For now, we focus on the case $M=R$.
	
	\begin{remark}
		The following statements are easy to observe,
		\begin{enumerate}
			\item If $T(R,R)$ is i-reversible then $R$ is i-reversible.
			\item If $R$ is commutative then $T(R,R)$ is commutative (hence i-reversible). 
			\item If $R$ is reduced then $T(R,R)$ is reversible \cite[Proposition~1.6]{kim2003extensions} (hence i-reversible).
		\end{enumerate}
		
	\end{remark}
	\noindent At this stage, we ask whether $T(R,R)$ is i-reversible when $R$ is i-reversible.  In view of the above remark this answer is trivial if $R$ is commutative or reduced. Hence, it is natural to ask the same question when $R$ is neither commutative nor reduced.\\
	We now provide two examples of non-commutative non-reduced rings whose trivial extensions are not i-reversible.
	
	\begin{eg}\label{2.3}
		Let $\mathbb H$ be the ring of quaternions, $S=T(\mathbb{H},\mathbb{H})$ and $R=T_2(S)$. We have, $e=$ 
		$\begin{bmatrix}
		E_{11}&0\\
		0&E_{11}
		\end{bmatrix}$,  where $E_{11}=\begin{bmatrix}
		1_
		S&0\\
		0&0
		\end{bmatrix}$  is a non-trivial idempotent.  Suppose $T(R,R)$ is i-reversible. Then, the corner ring $e\,T(R,R)\,e$ is reversible (see \cite[Proposition 2.1(4)]{khurana2020reversible}). However, $e\,T(R,R)\,e \cong T(S,S)$, which is not reversible (see \cite[Example 1.7]{kim2003extensions}). This is a contradiction. Therefore, $T(R,R)$ is not i-reversible.
	\end{eg}
	
	\begin{eg}\cite[Example~1.10]{HR}\label{2.4}
		Let $R=T(\mathbb{H}\times\mathbb{H},\mathbb{H}\times\mathbb{H})$. Since $\mathbb{H}\times\mathbb{H}$ is a reduced, $R$ is reversible (see \cite[Proposition 1.6]{kim2003extensions}). \\ Let $\alpha=\begin{bmatrix}
		\begin{bmatrix}
		(1,0)&(0,i)\\
		(0,0)&(1,0)
		\end{bmatrix} &\begin{bmatrix}
		(0,j)&(0,0)\\
		(0,0)&(0,j)
		\end{bmatrix}\\
		\begin{bmatrix}
		(0,0)&(0,0)\\
		(0,0)&(0,0)
		\end{bmatrix}&\begin{bmatrix}
		(1,0)&(0,i)\\
		(0,0)&(1,0)
		\end{bmatrix}
		\end{bmatrix}$  and $\beta=\begin{bmatrix}
		\begin{bmatrix}
		(1,0)&(0,1)\\
		(0,0)&(1,0)
		\end{bmatrix} &\begin{bmatrix}
		(0,k)&(0,0)\\
		(0,0)&(0,k)
		\end{bmatrix}\\
		\begin{bmatrix}
		(0,0)&(0,0)\\
		(0,0)&(0,0)
		\end{bmatrix}&\begin{bmatrix}
		(1,0)&(0,1)\\
		(0,0)&(1,0)
		\end{bmatrix}
		\end{bmatrix}$.
		\vspace{1mm}
		\\
		It is clear that $\alpha\beta$ is a non-zero idempotent. However, $\beta\alpha$ is not an idempotent. Therefore, $T(R,R)$ is not i-reversible.
	\end{eg}
\begin{remark}
In Example \ref{2.3}, the ring $R$ has a non-trivial idempotent $e$ which is not central. Whereas in Example \ref{2.4}, the ring $R$ has non-trivial idempotents and is also abelian. It thus becomes natural to impose the condition that $R$ has only trivial idempotents in which case, $R$ is also i-reversible and abelian.
\end{remark}

\begin{theorem}\label{P4.14}
		Let $R$ be a ring. 
		If $R$ has only trivial idempotents then $T(R,R)$ has only trivial idempotents and hence is i-reversible. 
	\end{theorem}
	\begin{proof}
		Since $R$ has only trivial idempotents, $R$ is abelian. We first show that if $R$ is abelian, every idempotent of $T(R,R)$ is of the form
		$\begin{bmatrix}
		a & 0\\
		0 & a
		\end{bmatrix}$, where $a$ is an idempotent in $R$ . Let $\alpha$ =
		$\begin{bmatrix}
		a&b\\
		0&a
		\end{bmatrix}$ be an idempotent in $T(R,R)$. A straight forward computation of $\alpha^2=\alpha$ yields  $a^2=a$ and $ab+ba=b$. Thus, $a$ is an idempotent in $R$. Since $R$ is an abelian ring with idempotent $a$, we get
		$(2a-1)b=0$. Now, $(2a-1)^2=4a^2-4a+1=1$, which shows $2a-1$ is a unit. Therefore, $b=0$. Hence  $\alpha$ = $\begin{bmatrix}
		a&0\\
		0&a
		\end{bmatrix}$, where $a$ is an idempotent in $R$. However, since $R$ has only trivial idempotents the only choices for $a$ are $0$ and $1$. Therefore,  $T(R,R)$ is a ring with only trivial idempotents and hence is i-reversible.
	\end{proof}
	\noindent In view of the above theorem, we may ask, does the i-reversibility of $T(R,R)$ ensure that $R$ has only trivial idempotents? The following example answers this question in the negative. 
	
	\begin{eg}\label{4.16}
		\noindent Let $S$ be a reduced ring with only trivial idempotents and $R=S\bigoplus S$. Then $R$ is a reduced ring with $(1,0)$ and $(0,1)$ as the only non-trivial idempotents. Since $R$ is reduced, $T(R,R)$ is reversible (see \cite[Proposition~1.6]{kim2003extensions}). Hence, $T(R,R)$ is i-reversible.
	\end{eg}

	\begin{cor}\label{P4.17}
		Let $S$ be a  ring and $R=T(S,S)$. 
		If $S$ has only trivial idempotents then $T(R,R)$ is i-reversible. 
	\end{cor}
	\begin{proof}
		Firstly, since $S$ has only trivial idempotents, $S$ is abelian. This implies that every idempotent of $R = T(S,S)$ is of the form
		$\begin{bmatrix}
		a & 0\\
		0 & a
		\end{bmatrix}$. However, since $S$ has only trivial idempotents, the only choices for $a$ are $a=0$ or $a=1$. Thus, $R$ has only trivial idempotents. Therefore, $T(R,R)$ is i-reversible.
	\end{proof}

\begin{theorem}\label{TH!}
Let $R$ be a ring with a non trivial central idempotent $e$. If $T(R,R)$ is i-reversible then $R$ is reversible.
\end{theorem}
\begin{proof}
Let $T(R,R)$ be i-reversible. Then $R$ is also i-reversible. Since $e$ is a non-trivial central idempotent in $R$, we can conclude that $R$ is reversible (see  \cite[Proposition~2.1.(4)]{khurana2020reversible}).
\end{proof}
\begin{cor}\label{THM!!}
Let $R$ be an abelian ring. If $T(R,R)$ is i-reversible then $R$ is either reversible or $R$ has only trivial idempotents.
\end{cor}
\begin{proof}
The proof is clear in view of Theorem \ref{TH!} and the fact that $R$ is abelian. 
\end{proof}

We characterize the i-reversibility of the trivial extension of rings which are abelian but not reversible.
\begin{theorem}\label{THM!!!!!}
Let $R$ be an abelian ring which is not reversible. The ring $T(R,R)$ is i-reversible if and only if $R$ has only trivial idempotents.
\end{theorem}
\begin{proof}
If $T(R,R)$ is i-reversible, then since $R$ is not reversible, by the previous corollary, $R$ has only trivial idempotents .
For the converse, see Theorem \ref{P4.14}.
\end{proof}

	\section{ Some Subrings of Triangular Matrix Rings. }
In  \cite{khurana2020reversible}, A.Khurana and D.Khurana show that for a ring $R$, $M_n(R)$ is not i-reversible for all $n > 2$. In this section, we show that $D_n(R)$ is an i-reversible subring of $M_n(R)$ when $n \geq 3$ and $R$ has only trivial idempotents.  Further, we prove that $D_n(R)$ is a maximal i-reversible subring of $T_n(R)$, when $R$ is a field and $n \geq 5$. Also, for the cases $n=3$ and $4$, we show that $D_n(R)$ is not a maximal i-reversible subring of $T_n(R)$ when $R$ is a field. Finally, in these two cases, we provide maximal i-reversible subrings of $T_n(R)$ strictly containing $D_n(R)$.
	\begin{theorem}\label{P4.7}
		Let $R$ be a ring and
		$n \geq 3$ be a positive integer. The ring $D_n(R)$ is i-reversible if and only if $R$ has only trivial idempotents. 
	\end{theorem}
	\begin{proof}
		Let $D_n(R)$ be i-reversible. If possible, let $R$ have a non-trivial  idempotent $e$. Let $ \alpha$ be the diagonal matrix with diagonal entries as $e$. Since $\alpha$ is a non trivial idempotent and $D_n(R)$ is i-reversible, the corner ring $\alpha D_n(R) \alpha$ is reversible (see \cite[Proposition~2.1.(2)]{khurana2020reversible} ). This is a contradiction since it is clear that $\alpha D_n(R) \alpha$ is not reversible (for example, $\alpha E_{12} \alpha \cdot \alpha E_{2n} \alpha \neq 0$ and $\alpha E_{2n} \alpha \cdot\alpha E_{12} \alpha =0$, where $E_{ij}$ is a matrix with $ij^{th}$ entry $1$ and other entries $0$).
		\\
		Conversely, let $R$ have only trivial idempotents.  Let  $P$ be an idempotent in $D_n(R)$ with $p$ as the diagonal entries.  Since $P^2=P$, a simple computation yields $p^2=p$. However, since $R$ has only trivial idempotents $p=0$ or $p=1$. If $p =1$ then $P$ is invertible (since $P=I+N$ for some nilpotent $N$). Also, since $P^2=P$ we have $P=I$. If $p=0$ then  $P$ is a nilpotent (since $P^{n}=0$). Since $P$ is also an idempotent, we arrive at $P=0$. Therefore $D_n(R)$ has only trivial idempotents and hence is i-reversible.\end{proof}
	 \begin{remark}\label{rr}
	     If $n=2$, then $D_n(R)=T(R,R)$.  This is i-reversible when $R$ has only trivial idempotents as discussed in Theorem \ref{P4.14}.
	 \end{remark}
	 
\begin{cor}\label{VN}
Let $R$ be a ring and $n$ be a positive integer. If $R$ has only trivial idempotents then the ring $V_n(R)=\lbrace (a_{ij})\in T_n(R) \ |\  a_{ij} =a_{(i+1) (j+1)}\  \forall \ i,j =1,2,\cdots \ n-1\rbrace $ is i-reversible.
\end{cor}
\begin{proof}
The result is trivial for $n=1$. For $n=2$, we have $V_n(R)=T(R,R)$. Therefore, by Theorem \ref{P4.14}, $V_n(R)$ is   i-reversible in this case. For $n \geq 3$, $V_n(R)$ is a subring of $D_n(R)$. Since $R$ has only trivial idempotents, by Theorem \ref{P4.7}, $D_n(R)$ is i-reversible. So  $V_n(R)$ is i-reversible.
\end{proof}
\begin{remark}
The converse of the above corollary is not true. For instance, if $R$ is a reduced ring containing a non-trivial idempotent, then by \cite[Theorem~2.5]{kim2003extensions}, we have $R[x]/(x^n)$ is reversible and hence, i-reversible.
It is not hard to see that for each positive integer $n$,  $V_n(R) \cong R[x]/(x^n)$. 
\end{remark}
\begin{theorem}
Let $R$ be a ring and let $n$ be a positive integer greater than or equal to $3$. The ring $T(D_n(R),D_n(R))$ is i-reversible if and only if $R$ has only trivial idempotents. 
\end{theorem}
\begin{proof}
Suppose $R$ has only trivial idempotents then $D_n(R)$ is abelian (see \cite[Corollary 3.5]{Characterizations}). However, $D_n(R)$ is not reversible (see \cite[Example 1.3]{kim2003extensions} and \cite[ Example 1.5]{kim2003extensions}). Therefore by Theorem \ref{THM!!!!!}, $T(D_n(R),D_n(R))$ is i-reversible.\\
Conversely, suppose $R$ has a non-trivial idempotent. Then by Theorem \ref{P4.7}, $D_n(R)$ is not i-reversible, and so, $T(D_n(R),D_n(R))$ is not i-reversible.
\end{proof}	
We now ask the following question.
 \begin{ques}
Given any positive integer $n \geq 3$ and any ring $R$ with only trivial idempotents, can one say $D_n(R)$ is a maximal i-reversible subring of $T_n(R)$?
\end{ques}
We answer this question in the affirmative if $n \geq 5$  and in the negative if  $n=3$ or $n=4$, when $R$ is a field. Furthermore, for the cases $n=3,4$ and a field $R$, we provide a suitable maximal i-reversible subring of $T_n(R)$ that strictly contains $D_n(R)$.

We first prove the following result which is valid for any ring $R$.
\begin{theorem} \label{TH3.4}
Let $R$ be any ring and $n \geq 5$. Suppose $S$ is a ring such that $D_n(R) \subset S \subseteq T_n(R)$ and $S$ has a non-trivial idempotent. Then $S$ is not i-reversible.
\end{theorem}
\begin{proof}
If $R$ has a non-trivial idempotent, then by Theorem \ref{P4.7}, $D_n(R)$ itself is not i-reversible and so $S$ is not i-reversible. So it is enough to prove the result when $R$ is a ring with only trivial idempotents. Let $A=[a_{ij}]$ be a non-trivial idempotent in $S$. Let $A'= A-\text{diag}(A)$, where $\text{diag}(A)$ is the diagonal matrix with diagonal entries $a_{ii}$. The matrix $A' \in D_n(R)$, and so $A' \in S$. Since $S$ is a ring, we have $\text{diag}(A) \in S$. The fact that $A$ is an idempotent in $S$ implies each $a_{ii}=0 \ \text{or} \ 1$. Moreover, since $A$ is a non-trivial idempotent, there exists positive integers $1 \leq i_1 < i_2 < \cdots < i_k \leq n $ such that $A-A'= E_{i_1i_1}+E_{i_2i_2}+\cdots +E_{i_ki_k} \neq I$. As $n \geq 5$, we claim that there is at least one non-trivial idempotent of the form  $E_{i_1i_1}+E_{i_2i_2}+\cdots +E_{i_ki_k} \in S$, for which $k\geq 3$. If $k=1$, then $E_{i_1i_1}\in S$, and in this case, $I-E_{i_1i_1} \in S$ will be the required idempotent. If $k=2$, then $E_{i_1i_1}+E_{i_2i_2}\in S$, and in this case, $I-(E_{i_1i_1}+E_{i_2i_2}) \in S$ will be the required idempotent.\\
\noindent Now consider $\alpha = E_{i_1i_1}+E_{i_2i_2}+\cdots +E_{i_ki_k} \in S$, where $k \geq 3$. Then $\alpha \cdot E_{i_1i_2}\cdot \alpha = E_{i_1i_2}$ and  $\alpha \cdot E_{i_2i_k}\cdot \alpha = E_{i_2i_k}$. It is clear that $E_{i_2i_k}\cdot E_{i_1i_2}= 0$ and $E_{i_1i_2}\cdot E_{i_2i_k}=E_{i_1i_k}$. Hence $\alpha S\alpha$ is not reversible. Therefore by \cite[Proposition~2.1.(2)]{khurana2020reversible}, $S$ is not i-reversible.
\end{proof}
\begin{remark}\label{R3.6}
When $n=4$, using the same argument as in the above theorem, one can easily see that if $S$ contains an idempotent of the form $E_{i_1i_1}+E_{i_2i_2}+\cdots +E_{i_ki_k}$ where $k=1$ or $k=3$, then $S$ is not i-reversible.
\end{remark}
\begin{theorem}\label{TH3.7}
Let $F$ be a field and $n\geq 3$. Let $S$ be a subring of $T_n(F)$ that strictly contains $D_n(F)$. Then $S$ contains a non-trivial idempotent.
\end{theorem}
\begin{proof}
 Let $A \in S \setminus D_n(F) $ and $B= \text{diag}(A)$ be the corresponding diagonal matrix. By the same arguments as in Theorem \ref{TH3.4}, $B\in S$. Let $b_1,b_2,\cdots,b_k$ be the distinct diagonal entries of $B$. This implies that the minimal polynomial of $B$ is of degree $k$. Therefore, the set  $\lbrace I, B, B^2, \cdots, B^{k-1} \rbrace$ is a linearly independent subset of $S$. Now, for each $m$ such that $1 \leq m \leq k$, let $E^{(m)}=[e^{(m)}_{ij}]$ be the diagonal matrix with $e^{(m)}_{ii}=$ 
 $\begin{cases}
1 ,\ \text{if} \  B_{ii}=b_m\\ 
0, \ \text{if} \ B_{ii} \neq b_m
\end{cases}$, where $B_{ii}$ is the $ii^{th}$ entry of $B$. It is clear that the set  $\lbrace E^{(1)},E^{(2)},\cdots, E^{(k)} \rbrace$ is linearly independent. Let $S'=\text{span}\lbrace E^{(1)},E^{(2)},\cdots, E^{(k)}\rbrace$. It is not hard to see that the set $\lbrace I,B,B^2,\cdots , B^{k-1}\rbrace \subset S'$, and so, $\text{span}\lbrace I,B,B^2,\cdots , B^{k-1}\rbrace \subseteq S'$. As both $\text{span}\lbrace I,B,B^2,\cdots , B^{k-1}\rbrace $ and $S'$ are of dimension $k$ as vector spaces over $F$, we can conclude that $\text{span}\lbrace I,B,B^2,\cdots , B^{k-1}\rbrace = S'$. This implies $S'$ is a subset of $S$. Clearly each $E^{(m)} \in S$ is a non-trivial idempotent.
\end{proof}
\begin{theorem}\label{TH3.8}
Let $F$ be a field and $n \geq 5$. Then $D_n(F)$ is a maximal i-reversible subring of $T_n(F)$.
\end{theorem}
\begin{proof}
By Theorem \ref{P4.7}, $D_n(F)$ is i-reversible. Let $S$ be a subring of $T_n(F)$ such that $D_n(F) \subsetneqq S$. We have to show that $S$ is not i-reversible. By Theorem \ref{TH3.4}, it is enough to show that $S$ contains a non-trivial idempotent. But this is true from Theorem \ref{TH3.7}. This completes the proof.
\end{proof}
The following examples illustrate that Theorem \ref{TH3.4} is not true if $n=3\  \text{or} \ 4$.
\begin{eg}\label{EG1}
Let $F$ be a field and 
$S_3(F)= \left\lbrace\begin{bmatrix}
a_{ij}
	\end{bmatrix} \in T_3(F)\mid  a_{11}=a_{22} \right\rbrace$. It is clear that $S_3(F)$ is a subring of $T_3(F) $  such that $D_3(F) \subsetneqq S_3(F)  $. It is also clear that $S_3(F)$ has non-trivial idempotents (for example $E_{11}+E_{22}$). We claim that $S_3(F)$ is i-reversible. Suppose $E=[e_{ij}]$ is a non-trivial idempotent in $S_3(F)$. This implies either $e_{11}=1$ and $e_{33}=0$ or $e_{11}=0$ and $e_{33}=1$. Further, the fact that $E^2=E$ will imply $e_{12}=0$. This will mean, $E$ is either of the form, $$E_1=\begin{pmatrix}
1&0&e_{13}\\
		0&1&e_{23}\\
		0&0&0
	\end{pmatrix} \qquad\text{or}\qquad E_2 = \begin{pmatrix}
0&0&e_{13}\\
		0&0&e_{23}\\
		0&0&1
 	\end{pmatrix}.$$ For $A,B\in S_3(F)$, suppose $AB$ is an idempotent of the form $E_1$, then the product of the first $2\times 2$ blocks of $A$ and $B$ is identity and %$A=\begin{pmatrix}
% a_{11}&a_{12}&a_{13}\\
% 		0&a_{11}&a_{23}\\
% 		0&0&a_{33}
% 	\end{pmatrix}$ and $B=\begin{pmatrix}
% b_{11}&b_{12}&b_{13}\\
% 		0&b_{11}&b_{23}\\
% 		0&0&b_{33}
% 	\end{pmatrix}$ are such that $AB= \begin{pmatrix}
% 1&0&e_{13}\\
% 		0&1&e_{23}\\
% 		0&0&0
% 	\end{pmatrix}$. Then  $\begin{pmatrix}
% a_{11}&a_{12}\\
% 		0&a_{11}
% 	\end{pmatrix}$ $\begin{pmatrix}
% b_{11}&b_{12}\\
% 		0&b_{11}
% 	\end{pmatrix} = \begin{pmatrix}
% 	1&0\\
% 		0&1
% 	\end{pmatrix}$ 
and $a_{33}b_{33}=0$. This implies that the product of the first $2\times 2$ blocks of $B$ and $A$ is identity and %$\begin{pmatrix}
% b_{11}&b_{12}\\
% 		0&b_{11}
% 	\end{pmatrix}$ $\begin{pmatrix}
% a_{11}&a_{12}\\
% 		0&a_{11}
% 	\end{pmatrix} = \begin{pmatrix}
% 	1&0\\
% 		0&1
% 	\end{pmatrix}$ 
and $b_{33}a_{33}=0$. So, $BA$ is again of the form $E_1$.  %$\begin{pmatrix}
% 1&0&e'_{13}\\
% 		0&1&e'_{23}\\
% 		0&0&0
% 	\end{pmatrix}$. 
Hence, $BA$ is an idempotent in $S_3(F)$. Again if $A$ and $B$ are such that $AB$ is an idempotent of the form $E_2$ then, the product of the first $2\times 2$ blocks of $A$ and $B$ is the zero matrix and %\begin{pmatrix}
% 0&0&e_{13}\\
% 		0&0&e_{23}\\
% 		0&0&1
% 	\end{pmatrix}$. In this case, $\begin{pmatrix}
% a_{11}&a_{12}\\
% 		0&a_{11}
% 	\end{pmatrix}$ $\begin{pmatrix}
% b_{11}&b_{12}\\
% 		0&b_{11}
% 	\end{pmatrix} = \begin{pmatrix}
% 	0&0\\
% 		0&0
% 	\end{pmatrix}$ 
 and $a_{33}b_{33}=1$. Now since $F$ is reduced, by \cite[Proposition~1.6]{kim2003extensions}, $T(F,F)$ is reversible. So, the product of the first $2\times 2$ blocks of $B$ and $A$ is the zero matrix. %$\begin{pmatrix}
% b_{11}&b_{12}\\
% 		0&b_{11}
% 	\end{pmatrix}$ $\begin{pmatrix}
% a_{11}&a_{12}\\
% 		0&a_{11}
% 	\end{pmatrix} = \begin{pmatrix}
% 	0&0\\
% 		0&0
% 	\end{pmatrix}$. 
Also $a_{33}b_{33}=1$ implies $b_{33}a_{33}=1$. Therefore, $BA$ is a matrix of the form $E_2$. 
% $\begin{pmatrix}
% 0&0&e'_{13}\\
% 		0&0&e'_{23}\\
% 		0&0&1
% 	\end{pmatrix}$.
	Hence $BA$ is an idempotent in this case as well. This proves that $S_3(F)$ is i-reversible.
\end{eg}

\begin{eg}\label{EG2}
Let $F$ be a field and 
$S_4(F)= \left\lbrace\begin{bmatrix}
a_{ij}
	\end{bmatrix}\in T_4(F) \mid a_{11}=a_{22}\text{ and }a_{33}=a_{44} \right\rbrace$. It is clear that $S_4(F)$ is a subring of $T_4(F) $  such that $D_4(F) \subsetneqq S_4(F)  $, and that $S_4(F)$ has non-trivial idempotents (for example $E_{11}+E_{22}$). We claim that $S_4(F)$ is i-reversible. Suppose $E=[e_{ij}]$ is a non-trivial idempotent in $S_4(F)$. This implies either $e_{11}=1$ and $e_{33} =0$ or $e_{11}=0$ and $e_{33}=1$. Further, the fact that $E^2=E$ will imply $e_{12}=0$ and $e_{34}=0$. Therefore, $E$ is either of the form  $$E_1=\begin{pmatrix}
1&0&e_{13}&e_{14}\\
		0&1&e_{23}&e_{24}\\
		0&0&0&0\\
		0&0&0&0
	\end{pmatrix} \qquad\text{or}\qquad
	E_2=\begin{pmatrix}
0&0&e_{13}&e_{14}\\
		0&0&e_{23}&e_{24}\\
		0&0&1&0\\
		0&0&0&1
	\end{pmatrix}.$$     A similar argument as in the previous example will now show that $S_4(F)$ is i-reversible.
\end{eg}
\begin{remark}
The above examples stand valid even if the base ring is any reduced ring with trivial idempotents instead of a field. However, we need the base ring to be a field in order to establish the following results.
\end{remark}

\begin{theorem}\label{TRT}
Let $F$ be a field. Then $S_3(F)$ is a maximal i-reversible subring of $T_3(F)$.
\end{theorem}
\begin{proof}
Firstly, it is clear from Example \ref{EG1} that $S_3(F)$ is an i-reversible subring of $T_3(F)$. Now suppose $S$ is any subring of $T_3(F)$ that strictly contains $S_3(F)$. Then since $S$ is also a vector subspace of $T_3(F)$ and $\text{dim}_F(T_3(F))= \text{dim}_F(S_3(F))+1$, we can conclude that $S=T_3(F)$. Therefore, $S_3(F)$ is a maximal i-reversible subring of $T_3(F)$.
\end{proof}
\begin{remark}
From the above proof, we can actually conclude that $S_3(F)$ is a maximal subring of $T_3(F)$.
\end{remark}
\begin{theorem}\label{TRTT}
Let $F$ be a field. Then $S_4(F)$ is a maximal i-reversible subring of $T_4(F)$.
\end{theorem}
\begin{proof}
It is clear from Example \ref{EG2} that $S_4(F)$ is i-reversible. Let $S$ be a subring of $T_4(F)$ such that $S$ strictly contains $S_4(F)$. Let $A \in S\setminus S_4(F)$ and $B=\text{diag(A)}$. Let $b_{ii}$'s denote the diagonal entries of $B$. We now consider the following cases.\\
\textbf{Case I:} If $B$ has $3$ or $4$ distinct diagonal entries, then by the same arguments as in the proof of Theorem \ref{TH3.7} and Remark \ref{R3.6}, $S$ is not i-reversible. \\
\textbf{Case II:} If $B$ is such that $b_{11}=b_{22}=b_{33}=b $ and $b_{44}=a$, then clearly $B'= B-aI \in S$. This implies $(b-a)^{-1}I\cdot B'=E_{11}+E_{22}+E_{33} \in S$. So, by Remark \ref{R3.6}, $S$ is not i-reversible. A similar argument can be applied if $B$ is such that $b_{11}=b_{22}=b_{44}=b $ and $b_{33}=a$, $b_{11}=b_{33}=b_{44}=b $ and $b_{22}=a$ or $b_{22}=b_{33}=b_{44}=b $ and $b_{11}=a$.

\noindent \textbf{Case III:} If  $B$ is such that $b_{11}=b_{33}=a $ and $b_{22}=b_{44}=b$, then by similar arguments as in the previous case, we can conclude that $E
	_{11}+E_{33} \in S$. As $S$ contains $S_4(F) $, $E_{11}+E_{22} \in S$. So, $E_{11}=(E
	_{11}+E_{33})\cdot (E_{11}+E_{22}) \in S$. Therefore, by Remark \ref{R3.6}, $S$ is not i-reversible.\\
\noindent \textbf{Case IV:} If $B$ is such that $b_{11}=b_{44}=a $ and $b_{22}=b_{33}=b$, then by similar arguments as in the previous cases, we can conclude that $E
	_{11}+E_{44} \in S$. This implies $E_{11} \in S$. Again by Remark \ref{R3.6}, $S$ is not i-reversible.	
	
\end{proof}

	\section{Dorroh Extensions and Nagata Extensions.}
	In this section, we discuss the i-reversibility of Dorroh Extensions and Nagata Extensions.
	\begin{definition}
		Let $R$ be an algebra over a commutative ring $S$. The Dorroh extension of $R$ by $S$ is the ring $R \times S$ with the following operations:
		\begin{align*}
			&(r_1,s_1)+(r_2,s_2) = (r_1+r_2,s_1+s_2) \\
			&(r_1,s_1)\cdot(r_2,s_2) = (r_1r_2+s_1r_2+s_2r_1,s_1s_2).
		\end{align*}
		The Dorroh extension of $R$ by $S$ is a ring with unity $(0,1)$.
	\end{definition}
	
	\begin{theorem}\label{P4.19}
		Let $S$ be a commutative ring and $R$ be an algebra (not necessarily unital) over $S$ containing a non-zero central idempotent $e$. Let $D$ be the Dorroh extension of $R$ by $S$. Then $D$ is i-reversible if and only if $D$ is reversible.
	\end{theorem}
	\begin{proof}
		Let $D$ be i-reversible. 
		Clearly, $(e,0)$  is a non-trivial idempotent in $D$. Now, $(r,s)(e,0)=(re+se,0)= (e,0)(r,s) $, which makes $(e,0)$ a central idempotent. Therefore, $D$ is reversible (see \cite[Proposition~2.1.(4)]{khurana2020reversible}). The reverse implication is clear.
	\end{proof}
	
	The Dorroh extension is usually used to embed a non-unital ring into a  unital ring. However, Alwis studied the Dorroh extension of $\mathbb{Z}$ over $\mathbb{Z}$ (see \cite{de1994ideal}). Cannon and Neuerburg \cite{cannon} generalized these results to the Dorroh extension of $R$ over $\mathbb{Z}$ where $R$ is any unital ring (and hence a algebra over $\mathbb{Z}$). More recently, Alhribat et al. further generalized these results to Dorroh extensions of a unital algebra $R$ over a unital ring $S$ (see \cite{ALH} for details). The following theorem gives a necessary and sufficient condition for the i-reversibility of the Dorroh extension of a unital algebra $R$ over a commutative unital ring $S$.
	
	\begin{theorem}\label{NSD}
	Let $R$ be a unital algebra over a commutative unital ring $S$ and $D$ be the Dorroh extension of $R$ by $S$. The ring $D$ is i-reversible if and only if $R$ is reversible. 
	\end{theorem}
\begin{proof}
Since $S$ is a commutative ring with unity and $R$ is a unital algebra over $S$, by \cite[Proposition 1.3]{ALH}, we have $D \cong S\times R$.  We know that $S \times R$ is i-reversible if and only if $S$ and $R$ are reversible. As $S$ is commutative, we further have $S \times R$ is i-reversible if and only if $R$ is reversible.
\begin{remark}
Theorem \ref{NSD} can be thought of as a particular case of Theorem \ref{P4.19} where $e$ is the unity in $R$. Therefore, the condition that $R$ and $S$ are reversible in Theorem \ref{NSD} is also a necessary and sufficient condition for the reversibility of $D$.
\end{remark}
\end{proof}	
	\begin{definition}
		Let $R$ be a commutative ring, $M$ be an $R$-module and $\sigma$ be an endomorphism of $R$. Give
		$R\bigoplus M$ a (possibly non-commutative) ring structure with component-wise addition and multiplication defined by,  $$(r_{1}, m_{1}) \cdot (r_{2},m_{2})=
		\big(r_{1}r_{2}, \sigma(r_{1})m_{2}+r_{2}m_{1}\big)$$ where $r_{i}$ $\in$ $R$ and $m_{i}$ $\in$ $M$. This extension is called the Nagata	extension of $R$ by $M$ and $\sigma$.
	\end{definition}
	
	\begin{theorem}\label{P4.21}
		Let $R$ be a commutative ring and $N$ be the Nagata Extension of $R$ by $M$ and $\sigma$. If $R$ has only trivial idempotents then $N$ is i-reversible.
	\end{theorem}

	\begin{proof}
		Let  $\alpha \beta$ be an idempotent, where $\alpha$ = ($a_{1},b_{1}$), $\beta$ = ($a_{2},b_{2}$) $\in$ $N$. We have,\begin{align*}
			\alpha\beta=&(a_{1}a_{2},\sigma(a_{1})b_{2}+a_2b_1)\\
			(\alpha\beta)^2=&(a_{1}a_{2}a_{1}a_{2},\sigma(a_{1}a_{2})(\sigma(a_{1})b_{2}+a_2b_1)+a_{1}a_{2}(\sigma(a_{1})b_{2}+a_2b_1))\\
			\beta\alpha=&(a_{2}a_{1},\sigma(a_{2})b_{1}+a_1b_2)\\
			(\beta\alpha)^2=&(a_{2}a_{1}a_{2}a_{1},\sigma(a_{2}a_{1})(\sigma(a_{2})b_{1}+a_1b_2)+a_{2}a_{1}(\sigma(a_{2})b_{1}+a_1b_2))
		\end{align*}
		Comparing the first component of $\alpha \beta$ and $(\alpha \beta)^2$, we get that $a_{1}a_{2}$ is an idempotent in $R$.
		Hence, $a_{1}a_{2} = 0$ or $a_1a_2=1$. If $a_{1}a_{2} = 0$, we have $\sigma$($a_{1}a_{2})= 0$. This will imply $\alpha \beta = 0$ leaving nothing to prove. If $a_{1}a_{2} = 1$ and $\sigma$($a_{1}a_{2}) = 0$ then
		$\beta \alpha = (1,\sigma(a_2)b_1+a_1b_2) = (\beta\alpha)^2$.
		So, in this case $N$ is i-reversible. Finally, if $a_{1}a_{2} = 1$ and $\sigma$($a_{1}a_{2}) = 1$ then
		comparing the second component of $\alpha\beta$  and $(\alpha\beta)^2$, we have
		$\sigma$($a_1$)$b_2$+$a_2b_1 = 0$.
		Multiplying throughout by $\sigma$($a_2$), we get
		$b_2+\sigma(a_2)a_2b_1 = 0$. Further multiplying this equation by $a_1$, we get $b_2a_1+\sigma(a_2)b_1 = 0$.
		Therefore, $\beta \alpha = (1,0) =$ $(\beta\alpha)^2$ and $N$ is i-reversible.
	\end{proof}  
	\begin{remark}
		The converse of the above theorem is not true. For example if $\sigma=Id_R$, where $R$ is a commutative ring with non-trivial idempotents (say $\mathbb{Z}_6$) then the Nagata extension of $R$ by $R$ and $\sigma$ is just $T(R,R)$. The commutativity of $R$ will imply the commutativity of $T(R,R)$ and hence its i-reversibility. We further provide an example where $\sigma \neq Id_R$.
	\end{remark}
	\begin{eg}\label{P4.23}
		Let $D$ be a commutative domain of characteristic zero, and $R$ = $D \bigoplus D$ with component-wise addition and
		multiplication. Define $\sigma$ : $R$ $\rightarrow$ $R$  by $\sigma$($s, t$) = ($t, s$), then $\sigma$ is an automorphism of $R$. Since $D$ is a domain, the only idempotents in $R$ are $(0,0)$, $(1,0)$, $(0,1)$ and $(1,1)$. Let $N$ denote the Nagata extension of $R$ by $R$ and $\sigma$. Let $\alpha=(r_1,m_1)$,  $\beta=(r_2,m_2) \in N$ for some $r_1,m_1,r_2,m_2 \in R$ such that $\alpha \beta$ is a non-zero idempotent. Now , $(\alpha \beta)^2 = \alpha \beta$  implies  $$((r_1r_2)^2,\sigma(r_1r_2)(\sigma(r_1)m_2+r_2m_1)+r_1r_2(\sigma(r_1)m_2+r_2m_1)) = (r_1r_2,\sigma(r_1)m_2+r_2m_1)$$ and therefore, $r_1r_2$ is an idempotent in $R$. If $r_1r_2 =(0,0)$ it gives $(\alpha\beta)^2=0$ and hence $\alpha\beta =0$. This  contradicts the fact that $\alpha\beta$ is a non-zero idempotent. If $r_1r_2 =(1,0) \ \text{or} \  (0,1)$, a direct computation gives $\beta\alpha = (\beta \alpha)^2$. Finally, if $r_1r_2 = (1,1)$, letting $r_1=(r_1',\bar{r_1})$, $r_2=(r_2',\bar{r_2})$, $m_1=(m_1',\bar{m_1})$ and $m_2=(m_2',\bar{m_2})$ we get,
		\begin{align*} 
			\beta\alpha &= \big((1,1),(\bar{r_2}m_1'+r_1'm_2',r_2'\bar{m_1}+\bar{r_1} \bar{m_2}))\\   (\beta\alpha)^2 &= ((1,1),2(\bar{r_2}m_1'+r_1'm_2',r_2'\bar{m_1}+\bar{r_1} \bar{m_2})). 
		\end{align*}
		However, since  $\alpha \beta$ is an idempotent, comparing the second component of $\alpha \beta$ and $(\alpha \beta)^2$ gives
		\begin{align}
			\bar{r_1}m_2'+r_2'm_1' =0.\label{Eq4.1}
		\end{align}
		Multiplying (\ref{Eq4.1}) from the left side with $(\bar{r_2} r_1')$ yields
		\begin{align}
			r_1'm_2'+\bar{r_2}m_1' =0. \label{Eq4.2}
		\end{align}
		Again, multiplying (\ref{Eq4.1}) from the left side with $(r_2'\bar{r_1} )$ we get
		\begin{align}
			\bar{r_1}\bar{m_2}+r_2'\bar{m_1} = 0. \label{Eq4.3}
		\end{align}
		Adding (\ref{Eq4.2}) and (\ref{Eq4.3}), we get 
		\begin{align*}
			r_1'm_2'+\bar{r_2}m_1'+\bar{r_1}\bar{m_2}+r_2'\bar{m_1}=0.  
		\end{align*}
		Therefore, $(\beta\alpha)^2$ = $(\beta\alpha)$ and hence $N$ is i-reversible.

	\end{eg}

	\section{I-reversibility of Polynomial and Laurent Polynomial Rings.}
	In  \cite[Proposition~ 2.1(7)]{khurana2020reversible} and \cite[Example~1.10]{jung2019reversibility}, the authors provide an example of a reversible ring $S$ such that $S[x]$ and $S[x,x^{-1}]$ are not i-reversible. In this section, we provide an example of a non-reversible, i-reversible ring $S$ such that $S[x]$ and $S[x,x^{-1}]$ are not i-reversible. We further give two independent sufficient conditions for the polynomial ring and Laurent polynomial ring to be i-reversible.  
	
We need the following result from \cite{khurana2020reversible}, 
	\begin{prop}\label{P5.1}\cite[Corollary~3.2]{khurana2020reversible}
		For any ring $R$ and any $n>1$, $T_n(R)$ is i-reversible if and only if $n=2$ and $R$ is reversible and has only trivial idempotents.     
	\end{prop}
	
	We give an example of a non-reversible, i-reversible ring $S$ such that $S[x]$ is not i-reversible.
	\begin{eg}\label{E3.17}
		In \cite[Example~2.1.]{kim2003extensions}, Kim and Lee provide a reversible ring $R$ such that $R[x]$ is not reversible. Further, Jung et.al in \cite[Example 2.1]{jung2019reversibility} show that $R$ has only trivial idempotents. By Proposition \ref{P5.1}, $S=T_2(R)$ is i-reversible. Clearly, $S$ is not reversible (for example $E_{12}E_{11}=0$ and $E_{11}E_{12}\neq 0$). Next, we show $S[x]$ is not i-reversible. Since $S[x]\cong \begin{bmatrix}
		R[x]&R[x]\\
		0&R[x]
		\end{bmatrix}$,  if $S[x]$ were i-reversible then $R[x]$ would be reversible (by Proposition \ref{P5.1}),  which is a contradiction. Hence, $S[x]$ is not i-reversible. Also, $S[x,x^{-1}] $  is not i-reversible since any subring of an i-reversible ring is  i-reversible.
	\end{eg}
	
We need the following result from \cite{kanwar2013idempotents},

\begin{theorem}\cite[Theorem~5]{kanwar2013idempotents}
For a ring $R$ the following conditions are equivalent,
\begin{enumerate}
    \item $R$ is abelian.
    \item Idempotents of $R$ commute with units of $R$.
    \item The set of idempotents of $R[[ x ]]$ is equal to the set of idempotents of $R$.
    \item The set of idempotents of $R[x,x^{-1}]$ is equal to the set of idempotents of $R$.
    \item The set of idempotents of $R[x]$ is equal to the set of idempotents of $R$.
    \item There exists $n \geq 1$ such that $R[x]$ does not contain idempotents which are polynomials of degree $n$.
\end{enumerate}
\end{theorem}

	\begin{theorem}
		\label{P3.23}
		If a ring $R$ has only trivial idempotents then $R[x]$, $R[[x]]$ and $R[x,x^{-1}] $ are i-reversible. 
	\end{theorem}
	\begin{proof}
		Since, the only idempotents of $R$ are trivial, $R$ is an Abelian ring. Therefore, idempotents of $R[x]$, $R[[x]]$ and $R[x,x^{-1}] $ are exactly idempotents of $R$ (see \cite[Theorem~5]{kanwar2013idempotents}).  Therefore $R[x]$, $R[[x]]$ and $R[x,x^{-1}] $ are rings with only trivial idempotents and hence i-reversible.
	\end{proof}
	
	\begin{theorem}
Let $R$ be an abelian ring. Then $R[x]$ is i-reversible if and only if $R[x,x^{-1}]$ is i-reversible.
\end{theorem}
\begin{proof}
The only case which is not obvious is when $R$ has a non-trivial idempotent and $R[x]$ is i-reversible. In this case, $R$ becomes reversible and the proof follows from \cite[Lemma~2.2.]{kim2003extensions}
\end{proof}
	\begin{definition}
	    A ring $R$ is said to be Armendariz if whenever polynomials $f(x)=a_0 +a_1x+\dots+a_mx^m$; $\ g(x)=b_0+b_1x+\dots+b_nx^n$ $\in R[x]$ satisfy $f(x)g(x)=0$, then $a_ib_j =0$ for each $i,j$.
	\end{definition}
	\begin{theorem}\label{T3.21}
		Let $R$ be an Armendariz ring, then the following statements are equivalent:\\
		(1) $R$ is i-reversible.\\
		(2) $R[x]$ is i-reversible.\\
		(3) $R[x,x^{-1}]$ is i-reversible.
	\end{theorem}
	\begin{proof}
		If $R$ is a ring with only trivial idempotents, the result is clear in view of \cite[Theorem~5.]{kanwar2013idempotents}. If  $R$ contains a non-trivial idempotent then $R$ is reversible (see \cite[Proposition~2.1.(4)]{khurana2020reversible}). The result is evident in view of \cite[Proposition~2.3]{kim2003extensions}.
	\end{proof} 
	\begin{eg}
	    Let $S$ be a commutative ring with only trivial idempotents. $R = T_2(S)$ is not Armendariz (see \cite[Example~1]{kim2000armendariz}). By Proposition \ref{P5.1} we have $R[x]\cong \begin{bmatrix}
		S[x]&S[x]\\
		0&S[x]
		\end{bmatrix}$ is i-reversible.
	\end{eg}
The following examples show that $R$ containing only trivial idempotents and $R$ being Armendariz are independent conditions.	
\begin{eg}
Let $S$ be a reduced ring with a non-trivial idempotent $e$ (for example $\mathbb{Z}_6$). Then the ring $R=T(S,S)$ is an Armendariz ring (see \cite[Corollary~4]{kim2000armendariz}) with non-trivial idempotents (for example the diagonal matrix with diagonal entries $e$).
\end{eg}
\begin{eg}\label{EXAMPLE}
Let $D$ be a domain, $S=T(D,D)$ and $R=T(S,S)$ which is a ring with only trivial idempotents (see Theorem \ref{P4.14}). However, $R$ is not Armendariz (see \cite[Example~5]{kim2000armendariz}).
\end{eg}
	\section{ Skew Polynomial Rings }
In this section, we ask whether Theorem \ref{P3.23} and Theorem  \ref{T3.21} generalize to the case of skew polynomial rings($R[x;\sigma]$). We also provide conditions on $R$ and $\sigma$ for which $R[x;\sigma]$ is not i-reversible.	

\begin{definition}
	Let $R$ be a ring and $\sigma$ be an endomorphism of $R$ such that $\sigma(1)=1$. We construct a skew polynomial ring by stipulating that for $b$ $\in$ $R$, $xb=\sigma(b)x$ on the set of left polynomials over $R$  and multiplication between elements in $R$ defined in the usual sense. The ring of skew polynomials over $R$ is denoted by $R[x;\sigma]$.
\end{definition}
\begin{remark}
The following statements are easy to observe,
\begin{enumerate}
    \item Skew polynomial rings can be similarly constructed on the set of right polynomials over $R$ by stipulating $bx=x\sigma(b)$. In this case, we can observe that $R[x,\sigma]/(x^2)$ is isomorphic to the Nagata extension of $R$ by $R$ and $\sigma$.
    \item Suppose $R$ is a ring in which $\sigma(1)=1$ implies $\sigma=Id_R$ (for example $\mathbb{Z},\ \mathbb{Q},\  \mathbb{Z}_n$), then $R[x;\sigma]=R[x]$.
    
\end{enumerate}

\end{remark}
We first give a sufficient condition for the i-reversibility of $R[x;\sigma]$. For that, we need the following Lemma.
\begin{lem}\label{LEM}
 Let $R$ be a ring and $\sigma$ be an endomorphism of $R$ which fixes a central idempotent $e$. Then any idempotent in $R[x;\sigma]$ with constant term $e$ has to be $e$ itself.
\end{lem}
\begin{proof}
Let $p(x)=e+d_1x+d_2x^2+\cdots+d_kx^k \in R[x;\sigma]$ be an idempotent. Since $(p(x))^2=p(x)$ and $e$ is a central idempotent fixed by $\sigma$, comparing the coefficient of $x$ in $(p(x))^2$ and $p(x)$, we get $(2e-1)d_1=0$. Since $2e-1$ is a unit, we get $d_1=0$. Similarly, for all $i\geq 2$, comparing the coefficients of $x^i$ recursively gives $d_2=d_3=\cdots=d_k=0$. Therefore, $p(x)=e$.
\end{proof}

\begin{theorem}\label{T3.29}
Let $R$ be a ring with only trivial idempotents then $R[x;\sigma]$ is i-reversible.
\end{theorem}
\begin{proof}
 Let $f(x)g(x)$ be a non-zero idempotent in $R[x;\sigma]$, where  $f(x)=\sum\limits_{i=0}^{n} a_ix^i$ and $g(x)=\sum\limits_{j=0}^{m} b_jx^j$. Then $a_0b_0$ is an idempotent in $R$. Since $R$ has only trivial idempotents the choices for $a_0b_0$ are $0$ and $1$. But by Lemma \ref{LEM} and the fact that $f(x)g(x) \neq 0$ we get, $a_0b_0=1$ and hence  $f(x)g(x)=1$. This proves that $g(x)f(x)$ is an idempotent making $R[x;\sigma]$ i-reversible.
\end{proof}

We now see conditions under which $R[x;\sigma]$ is not i-reversible.
% \begin{theorem}\label{T6.4}
% Let $R$ be a ring with non-trivial idempotents.
% If $\sigma(1)=0 $  or $\sigma(1)=e$ for some non-trivial idempotent $e$ then $R[x;\sigma]$ is not i-reversible.
% \end{theorem}
% \begin{proof}
%   Consider,  $f(x)=e+(1-e)x$ and $g(x)=1$. Now if $\sigma(1)=0 \ \text{or} \ e$, we have $f(x)g(x)=e$, a non-zero idempotent and $g(x)f(x)=e+(1-e)x$ which is clearly not an idempotent in $R[x;\sigma]$. Therefore, $R[x;\sigma]$ is not i-reversible. 
% \end{proof}
% The following result completely characterizes the i-reversibility of skew polynomial rings over $\mathbb{Z}_n$.
% \begin{theorem}\label{T6.5}
% The ring $\mathbb{Z}_n[x;\sigma]$ is i-reversible if and only if either $n=p^k$ for some prime $p$ or $\sigma(1)=1$.
% \end{theorem}
% \begin{proof}
% Suppose $\sigma(1) \neq 1$ and $n \neq p^k$ for any prime $p$. The fact that $n \neq p^k$ implies that $\mathbb{Z}_n$ has non-trivial idempotents. Also since $\sigma(1) \neq 1$ and $\sigma$ is an endomorphism of $\mathbb{Z}_n$, $\sigma(1)$ has to be either $0$ or a non-trivial idempotent $e$. By Theorem \ref{T6.4}, $\mathbb{Z}_n[x;\sigma]$ is not i-reversible.\\
% Conversely, if $n=p^k$ then $\mathbb{Z}_n$ has only trivial idempotents, and hence by Theorem \ref{T3.29}, $\mathbb{Z}_n[x;\sigma]$ is i-reversible. Also if $\sigma(1)=1$, then $\sigma$ is the identity endomorphism on $\mathbb{Z}_n$. So, $\mathbb{Z}_n[x;\sigma] = \mathbb{Z}_n[x]$, which is a commutative ring and hence i-reversible.
% \end{proof}
% The following theorem shows that even when $\sigma(1)=1$, $R[x;\sigma]$ need not be i-reversible.
\begin{theorem}\label{T}
Let $R$ be an i-reversible ring with a non-trivial central idempotent $e$. If $\sigma$ is a non-injective endomorphism of $R$ such that  $\sigma(e)=e $ then $R[x;\sigma]$ is not i-reversible.
\end{theorem}
\begin{proof}
Clearly, $e$ is a non-trivial idempotent in $R[x;\sigma]$. Also, $x^ke=\sigma^k(e)x^k =ex^k$ for any $k \in \mathbb{N}$ . Therefore, $e$ is a non-trivial central idempotent in $R[x;\sigma]$. Assume $R[x;\sigma]$ is i-reversible. The fact that $e$ is a non-trivial central idempotent implies $R[x;\sigma]$ is reversible (see \cite[Proposition~2.1.(4)]{khurana2020reversible}). As $\sigma$ is a non-injective, there exists $b \neq 0 \in R$ such that $\sigma(b)=0$. Clearly, $bx \neq 0$ and $xb=0$, which is a contradiction. Therefore, $R[x;\sigma]$ is not i-reversible.
\end{proof}

We now give two examples to show that the conditions on $\sigma$ from Theorem \ref{T} are not sufficient conditions for the non i-reversibility of $R[x;\sigma]$.

\begin{eg}
Let $R$ be the ring of real sequences with pointwise addition and multiplication as the operations. Let $\sigma$ be the endomorphism of $R$ given by $\sigma(a_1,a_2,a_3,a_4,\cdots)=(a_2,a_3,a_4,\cdots)$. Then $\sigma$ is clearly non-injective and does not fix any non-trivial idempotent in $R$. Consider the poynomials $f(x)=(0,1,1,0,0,0,\cdots)x+(0,1,0,0,\cdots)$ and $g(x)=(0,1,1,0,0,\cdots)$ in $R[x,\sigma]$. Then $fg$ is a non-zero idempotent but $gf=f$ is not. Thus $R[x,\sigma]$ is not i-reversible.
\end{eg}
% \begin{remark}
% If $n=p^k$ for some prime $p$, then we know that $\mathbb{Z}_n$ has only trivial idempotents. We also know that the idempotents of $\mathbb{Z}_n$ are same as those of $\mathbb{Z}_n[t]$. So, if $R=\mathbb{Z}_n[t]$ and $n=p^k$, then by Theorem \ref{T3.29}, for any endomorphism $\sigma$ on $R$, $R[x;\sigma]$ is i-reversible. %This is because, in this case, $\mathbb{Z}_n$ has only trivial idempotents and the idempotents of $\mathbb{Z}_n$ and $\mathbb{Z}_n[t]$ always coincide. 
% If $n \neq p^k$, then $R$ will have non-trivial idempotents. So, by Theorem \ref{T6.4}, for $R[x; \sigma]$ to be i-reversible, we are forced to have $\sigma(1)=1$. However, this is not sufficient to guarantee the i-reversibility of $R[x; \sigma]$. This is because, in this case, $\sigma$ could still be non-injective, and so, by the above theorem, $R[x;\sigma]$ will not be i-reversible. Therefore, if $n \neq p^k$, $\sigma(1)=1$ and injectivity of $\sigma$ are both necessary for the i-reversibility of $R[x; \sigma]$.
% It is also not hard to see that under these assumptions on $\sigma$, the i-reversibility of $R[x;\sigma]$ is equivalent to its reversibility, if $n=p^k$. 
% \end{remark}

\begin{eg}
Let $S$ be a ring with a non-trivial central idempotent $e$ and let $R=S\times S$. Let $\sigma$ be the endomorphism on $R$ such that $\sigma(a,b)=(b,a)$. It is not hard to see that $(e,e)$ is a non trivial central idempotent in $R[x, \sigma]$. Suppose, $R[x, \sigma]$ is i-reversible. By \cite[Proposition~2.1.(4)]{khurana2020reversible}, $R[x, \sigma]$ is reversible. This gives a contradiction since we have $f=(1,0)x$  and $g=(1,0)$
$\in$ $R[x, \sigma]$ such that $fg =0 $ however $gf \neq 0$. \end{eg}

The following example shows that unlike polynomial rings, $R$ being i-reversible and Armendariz is not a sufficient condition for i-reversibility of $R[x;\sigma]$.
\begin{eg}
Let $S$ be a reduced ring with a non-trivial central idempotent $e$. Let $R=T(S,S)$ with the endomorphism $\sigma$ defined by $\sigma\begin{bmatrix}
	a&b\\
	0&a\\
	\end{bmatrix}=\begin{bmatrix}
	a&0\\
	0&a\\
	\end{bmatrix}$. Since $S$ is reduced, $R$ is Armendariz (see \cite[Corollary~4]{kim2000armendariz}) and reversible (see \cite[Proposition~1.6]{kim2003extensions}). Using the fact that $\sigma$ fixes the non-trivial central idempotent$\begin{bmatrix}
	e&0\\
	0&e\\
	\end{bmatrix}$ and by Theorem \ref{T}, $R[x;\sigma]$ is not i-reversible.
\end{eg}

 We now provide a class of rings $R$ which may contain non-trivial idempotents yet, $R[x;\sigma]$ is i-reversible (c.f. Example \ref{RIGID}).

\begin{definition}
Let $R$ be a ring with an endomorphism $\sigma$. $R$ is called $\sigma$-Armendariz if for polynomials $p(x)=\sum\limits_{i=0}^ma_ix^i$ and $\ q(x)=\sum\limits_{j=0}^nb_jx^j$ $\in R[x;\sigma]$,  $p(x)q(x)=0$ implies $a_ib_j =0$ for all $0 \leq i\leq m$ and $0 \leq j\leq n$.
\end{definition}

\begin{theorem}\label{THMM}
Let $R$ be an i-reversible ring. If $R$ is  $\sigma$-Armendariz then $R[x;\sigma]$ is i-reversible.
\end{theorem}
\begin{proof}
Let $R$ be an i-reversible $\sigma$-Armendariz ring. If $R$ has only trivial idempotents then $R[x;\sigma]$ is i-reversible (see Theorem \ref{T3.29}). Let $R$ contain a non-trivial idempotent $e$. Now, $R$ is abelian since $R$ is $\sigma$-Armendariz,  (see \cite[Lemma~1.1]{kim2015alpha}). Hence $e$ is a central idempotent and $R$ is reversible (see \cite[Proposition~2.1.(4)]{khurana2020reversible}). Finally,  $R$ is a $\sigma$-Armendariz and reversible implies $R[x;\sigma]$ is reversible (see \cite[Theorem~5]{pourtaherian2011skew}) and hence i-reversible.
\end{proof}

\begin{definition}
Let $R$ be a ring with an endomorphism $\sigma$. Then $R$ is called $\sigma$-rigid if for any $a\in R$ with $a\sigma(a)=0$, we have $a=0$. 
\end{definition}

\begin{eg}\label{RIGID}
Let $S$ be an $\alpha$-rigid ring where $\alpha$ is an endomorphism of $S$ and let $R=S\times S$ with component-wise addition and multiplication. Since $S$ is an $\alpha$-rigid ring hence $S$ is reduced (see \cite{hong2000ore}). Therefore, $R$ is reversible. Consider the endomorphism $\sigma$ of $R$, defined by $ \sigma(a,b)=(\alpha(a),\alpha(b))$. In view of \cite[Example~1.4(2)]{kim2015alpha}, we have, $R$ is  $\sigma$-rigid and hence  $\sigma$-Armendariz (see \cite[Lemma~1.3]{kim2015alpha}). It is clear that $R$ contains non-trivial idempotents.

\end{eg}

\begin{eg}  \label{TI}
The ring $R$ in Example \ref{EXAMPLE} contains only trivial idempotents. However, $R$ is not $\text{Id}_R$-Armendariz.
\end{eg} 

\begin{remark}
    The above examples (Example \ref{RIGID} and Example \ref{TI} ) show that $R$ containing only trivial idempotents and $R$ being $\sigma$-Armendariz are independent conditions.
\end{remark}
\section*{Acknowledgments}
The authors would like to sincerely thank the referee for the comments. The comments were very helpful for us to explore the subject in a better way.
%\end{acknowledgement}
%\end{acknowledgement}
%	\bibliographystyle{ieeetr}
%	\bibliography{refs}
	
\end{document}